\documentclass{amsart}
\usepackage{amssymb,amstext,amsmath,amscd,amsthm,amsfonts,enumerate,graphicx,latexsym,stmaryrd,multicol,mathtools}

\usepackage{thmtools}
\usepackage{hyperref}
\usepackage[capitalise]{cleveref}  
\hypersetup{colorlinks=true} 

\usepackage[usenames]{color}
\usepackage[all]{xy}

\newtheorem{thm}{Theorem}[section]
\newtheorem{thmx}{Theorem}

\newtheorem{lem}[thm]{Lemma}
\newtheorem{prop}[thm]{Proposition}
\newtheorem{cor}[thm]{Corollary}
\newtheorem{quest}[thm]{Question}
\theoremstyle{definition}

\newtheorem{rem}[thm]{Remark}

\newtheorem*{ack}{Acknowledgements}

\theoremstyle{remark}

\numberwithin{equation}{thm}

\def\ass{\operatorname{Ass}}

\def\cm{\mathsf{CM}}

\def\depth{\operatorname{depth}}

\def\Ext{\operatorname{Ext}}

\def\ge{\geqslant}

\def\Lcal\Ical{\operatorname{LocInj}}

\def\Hom{\operatorname{Hom}}

\def\id{\operatorname{id}}

\def\m{\mathfrak{m}}

\def\mod{\operatorname{\mathsf{mod}}}

\def\n{\mathfrak{n}}

\def\ef{\mathsf{f}}
\def\p{\mathfrak{p}}
\def\pd{\operatorname{pd}}

\def\spec{\operatorname{Spec}}
\def\mspec{\operatorname{mSpec}}

\def\syz{\mathsf{\Omega}}

\def\Tor{\operatorname{Tor}}

\def\X{\mathcal{X}}

\def\ZZ{\mathbb{Z}}

\def\Mod{\operatorname{Mod}}

\def\Gpd{\operatorname{Gpd}}
\def\Gfd{\operatorname{Gfd}}
\def\CMpd{\operatorname{CMpd}}
\def\CMfd{\operatorname{CMfd}}
\def\Rpd{\operatorname{Rpd}}
\def\Rfd{\operatorname{Rfd}}

\def\Gcal{\mathcal{G}}
\def\Pcal{\mathcal{P}}
\def\Ical{\mathcal{I}}
\def\Fcal{\mathcal{F}}
\def\Ecal{\mathcal{E}}
\def\Lcal{\mathcal{L}}
\def\Ccal{\mathcal{C}}
\def\Dcal{\mathcal{D}}
\def\Mcal{\mathcal{M}}
\def\Rcal{\mathcal{R}}
\def\Scal{\mathcal{S}}

\def\fd{\operatorname{fd}}

\def\RHom{\mathbf{R}\mathsf{Hom}}

\def\GL{(\mathsf{GL})}
\def\FD{(\mathsf{FD})}
\def\LF{(\mathsf{LF})}
\begin{document}
\allowdisplaybreaks
\title[]{Govorov--Lazard and Finite deconstructibility for Gorenstein and restricted homological dimensions} 
\author[Souvik Dey, Michal Hrbek, Giovanna Le Gros]{Souvik Dey, Michal Hrbek, Giovanna Le Gros} 

\address{Souvik Dey: Department of Mathematical Sciences, University of Arkansas, 850 West Dickson Street Fayetteville, Arkansas 72701 United States} 
\address{Michal Hrbek: Institute of Mathematics of the Czech Academy of Sciences, Žitná 25,
115 67 Prague, Czech Republic}
\address{  Giovanna Le Gros: Faculty of Mathematics and Physics,
Department of Algebra,
Charles University, 
Sokolovsk\'{a} 83, 186 75 Praha, 
Czech Republic}
\email{souvikd@uark.edu, hrbek@math.cas.cz, giovanna.legros@matfyz.cuni.cz}  

\subjclass[2020]{13C60, 13C14, 13D05, 13D07}
\thanks{}
\begin{abstract}

Over Cohen–Macaulay rings admitting a pointwise dualizing module, we show that the class of modules of restricted projective dimension bounded by any integer is finitely deconstructible and that the class of modules of restricted flat dimension bounded by any integer satisfies the Govorov–Lazard property. Along the way, we prove the corresponding result for Gorenstein projective and flat dimensions over (locally) Gorenstein rings. Outside of Cohen–Macaulay rings, we consider analogous properties for restricted projective dimension zero and restricted flat dimension zero and establish them for commutative noetherian rings of finite Krull dimension. This has consequences for the corresponding classes of finitely generated modules being preenveloping in certain cases and provides generalizations of Holm’s results on structure of balanced big Cohen–Macaulay modules in various directions.


\end{abstract}
\maketitle
\section*{Introduction} 
It is desirable for a class of modules determined by a given homological condition to be built up in a tractable way from modules of finite presentation. A classical result of this kind due to Govorov \cite{Gov65} and Lazard \cite{Laz69} states that any flat module can be presented as a direct limit of finitely presented projective modules. Naturally, the \textit{Govorov--Lazard property} was studied for relative variants of flat dimensions. Enochs and Jenda \cite[10.3.8]{EJ11} showed that any Gorenstein flat module is a direct limit of finitely presented Gorenstein projective modules over Iwanaga--Gorenstein rings, while Beligiannis and Krause \cite{Beligiannis-thick-2008} and Holm and Jørgensen \cite{HJ11} showed that this fails in general for non-Gorenstein rings. Holm \cite{Hol17} also showed that Cohen--Macaulay flat modules enjoy the Govorov--Lazard property over any local Cohen--Macaulay ring admitting a dualizing module. 

Replacing flat dimension by the projective dimension of large modules, the resulting classes of modules are rarely closed under direct limits, and thus a different kind of deconstruction should be considered. We say that a class of modules is \textit{finitely deconstructible} if any module in it can be presented as a direct summand of a transfinite extension of finitely presented modules from the same class. An unimpressive but important case is the projective modules themselves, reflecting the simple fact that any projective module is a direct summand of a free one. However, the finite deconstructibility of the class $\Pcal_1$ of modules of projective dimension at most one is already a subtle problem, studied in detail by Bazzoni and Herbera \cite{BH09}. Also more recently, Positselski \cite{resol} demonstrated a general theory for countably presented versions of Govorov--Lazard and deconstructibility, establishing it for various dimensions under relatively mild assumptions.

The study of the classes we are interested in is inextricably intertwined with approximation theory. In this paper, we make strong use of approximations in many different forms. We will not further outline details of this relation here, however it worth mentioning how the Govorov--Lazard property can be used to show the existence of certain approximations. Under some natural assumptions, a class satisfies the Govorov--Lazard property and is closed under products if and only if its finitely presented constituents form a preenveloping class, a result due to Crawley-Boevey (cf. \cref{envelop1}), see also \cite[Theorem C]{Hol17}. 

From now on, we restrict to a commutative noetherian ring $R$. In the recent work \cite{hrbek2023finite}, it was shown that $R$ is Cohen--Macaulay if and only if all of the classes $\Pcal_n$ of modules of projective dimension bounded by an integer $n$ are finitely deconstructible. Also, the Govorov--Lazard property holding for all the classes $\Fcal_n$ of modules of bounded flat dimension was shown to characterize the somewhat wider class of ``almost Cohen--Macaulay'' rings. Our first main result is a Gorenstein analog of these.
\begin{thmx}[\cref{Gpd-deconstr}, \cref{Gfd-deconstr}]\label{ThmA}
Let $R$ be a (locally) Gorenstein ring. Then the following ``Govorov--Lazard'' and ``Finite Deconstructibility'' properties hold for Gorenstein flat and Gorenstein projective dimensions, respectively:
\begin{enumerate}
    \item[$\GL$] Any $R$-module of Gorenstein flat dimension at most $n$ is a direct limit of finitely generated $R$-modules of Gorenstein projective dimension at most $n$.
    \item[$\FD$] Any $R$-module of Gorenstein projective dimension at most $n$ is a direct summand of a transfinite extension of finitely generated $R$-modules of Gorenstein projective dimension at most $n$.
\end{enumerate}
\end{thmx}
If $R$ is not Gorenstein, both the assertions of \cref{ThmA} can fail, see \cref{R-Gor-equiv-cond}. Therefore, we consider instead the restricted flat and restricted projective dimensions as introduced by Christensen, Foxby, and Frankild \cite{CFF02}. If $R$ is Cohen--Macaulay with a (pointwise) dualizing module, these coincide with the Cohen--Macaulay flat and projective dimensions of Holm and Jørgensen \cite{HJ07}, see \cref{RfdCMfd}. Since the restricted dimensions recover the Gorenstein dimensions if $R$ is Gorenstein, the following generalizes \cref{ThmA} to several non-Gorenstein contexts.
\begin{thmx}\label{ThmB}[\cref{gor}, \cref{CMproj}, \cref{thm-acm}]
    Let $R$ be a commutative noetherian ring. Then:
    \begin{enumerate}
        \item Assume that $R$ is of finite Krull dimension. Then:
        \begin{enumerate}
            \item[(i)] The class of restricted flat modules satisfies the Govorov--Lazard property.
            \item[(ii)] The class of restricted projective modules satisfies the finite deconstructibility property.
        \end{enumerate}
        \item Assume that $R$ is Cohen--Macaulay with a pointwise dualizing module. Then for any $n \geq 0$:
        \begin{enumerate}
            \item[(i)] The class of modules of Cohen--Macaulay flat dimension bounded by $n$ satisfies the Govorov--Lazard property.
            \item[(ii)] The class of modules of Cohen--Macaulay projective dimension bounded by $n$ satisfies the finite deconstructibility property.
        \end{enumerate}
        \item Assume that $R$ is almost Cohen--Macaulay of finite Krull dimension. Then for any $n \geq 0$:
        \begin{enumerate}
            \item[(i)] The class of modules of restricted flat dimension bounded by $n$ satisfies the Govorov--Lazard property.
        \end{enumerate}
    \end{enumerate}
\end{thmx}
We introduce the general formulations of the Govorov--Lazard $\GL$ and Finite Deconstructibility $\FD$ properties in \cref{S:setup} and relate them to cotorsion theory. In \cref{S:restricted} we focus on classes of modules of bounded restricted homological dimensions and show that \cref{ThmB}(1) follows from a somewhat basic argument using the Baer criterion. The finite Krull dimension assumption is removed in \cref{SS:dualizing} under the existence of a dualizing module. \cref{ThmA} is proved in \cref{S:Gor}, the argument uses the results for classical homological dimensions of \cite{hrbek2023finite}. The proof of \cref{ThmB}(2) builds further on the Gorenstein case and occupies most of \cref{S:CMrings}. The final \cref{S:ACM} establishes a general criterion \cref{criterionRFn} for the Govorov--Lazard property for restricted flat dimension to hold using the recent classification of hereditary Tor-pairs of \cite{HHLG24}. We do not know if this criterion, and therefore the Govorov--Lazard property, holds for restricted flat dimension over any commutative noetherian ring, see \cref{Q:Rfd}.

\begin{ack}
All three authors were supported by the GAČR project 23-05148S. The first author was also supported by Charles University Research Centre
program No. UNCE/24/SCI/022. The second author was also supported by the Academy of Sciences of the Czech Republic (RVO 67985840). All of the work done in this paper took place when the first author was a research scholar at the Department of Algebra of Charles University, Prague, and he is very grateful for the outstanding atmosphere fostered by the department. 
\end{ack}

\section{Setup and preliminary results}\label{S:setup}
Throughout the paper, $R$ will always denote a ring which is commutative, noetherian, but not necessarily local nor of finite Krull dimension, unless so specified. All subcategories are full and closed under isomorphisms. By $\Mod R$ we denote the category of all $R$-modules and by $\mod R$ the subcategory of finitely generated $R$-modules. Given a subcategory $\Ccal$ of $\Mod R$ we let $\Ccal^{<\omega}$ denote the subcategory $\Ccal \cap \mod R$ of those $R$-modules in $\Ccal$ which are finitely generated. 

In the paper, we will study the following two properties each of which allows one to build objects in a class $\Ccal$ from those in $\Ccal^{<\omega}$:

\begin{enumerate}
    \item[$\GL$] Any module in $\Ccal$ is isomorphic to a direct limit of modules from $\Ccal^{<\omega}$.
    \item[$\FD$] Any module in $\Ccal$ is a direct summand in a transfinite extension of modules from $\Ccal^{<\omega}$.
\end{enumerate}
The first is an abbreviation for ``Govorov--Lazard'' property, after the classical theorem of Govorov and Lazard, which established $\GL$ for the class $\Fcal_0$ of all flat $R$-modules. Note that, by a standard argument, $\Ccal$ satisfies $\GL$ if and only if it has the following property: Any map $F \to C$ from a finitely presented module $F$ to $C \in \Ccal$ factorizes through a module from $\Ccal^{<\omega}$, see \cite[Proposition 2.1]{Lenzing83}. The second abbreviates ``finite deconstructibility'', a terminology borrowed from the theory of cotorsion pairs, see \cite[\S 8]{GT12}. A prototypical class satisfying $\FD$ is the class $\Pcal_0$ of all projective $R$-modules, as any projective module is a direct summand of a free one.

In our goal of studying $\GL$ and $\FD$ for the classes of modules satisfying a bound on various homological dimensions, it will be useful to also consider the following lifting property of a class $\Ccal$.
\begin{enumerate}
    \item[$\LF$] For any prime ideal $\p \in \spec R$ and any $M \in \Ccal \cap \mod {R_\p}$ there is $N \in \Ccal^{<\omega}$ such that $M$ is a direct summand of $N_\p$.
\end{enumerate}
There is a general connection between $\GL$ and $\LF$ providing by the following two Lemmas.
\begin{lem}\label{lift-GL}
    Let $\Ccal$ be a subcategory of $\Mod R$ which satisfies $\GL$. Then for any flat ring epimorphism $R \to S$ and any finitely generated $S$-module $M$ which belongs to $\Ccal$ as an $R$-module there is $N \in \Ccal^{<\omega}$ such that $M$ is a direct summand of $N \otimes_R S$.
    
    In particular, $\Ccal$ satisfies $\LF$.
\end{lem}
\begin{proof}
    By the assumption, $M \cong \varinjlim_{i \in I}M_i$, where $M_i \in \Ccal^{<\omega}$ for all $i \in I$. Applying the extension of scalars functor $- \otimes_R S$, we obtain the direct limit representation $M = \varinjlim_{i \in I}(M_i \otimes_R S)$ in $\Mod {S}$. As $M$ is finitely presented as an $S$-module, there is $i \in I$ such that the canonical map $(M_i \otimes_R S) \to M$ is a split epimorphism, which concludes the proof by setting $N = M_i$.

    The final claim follows from taking $R \to S$ to go through the localisations $R \to R_\p$ for $\p \in \spec R$.
\end{proof}

There is a sort of converse of \cref{lift-GL} for some well behaved classes. In particular, classes fitting into Tor-pairs satisfy the closure assumptions of \cref{LF-to-GL}, so the lemma can be used to prove a sort of local-to-global principle. In the following we let $\Ccal_\p \coloneq \{C_\p \mid C \in \Ccal\} \subseteq \Mod{R_\p}$ for any $\p \in \spec R$.
\begin{lem}\label{LF-to-GL}
    Let $\Ccal$ be a class closed under $\varinjlim$, direct summands, extensions, and is determined locally in the sense that: 
    $$C \in \Ccal \text{ if and only if } C_\m \in \Ccal \text{ for every } \m \in \mspec R.$$ 
    Moreover, assume that $\Ccal_\m$ satisfies $\GL$ in $\Mod R_\m$ and that $\Ccal$ satisfies $\LF$. Then $\Ccal$ satisfies $\GL$.

\end{lem}

\begin{proof}
    By \cite[Theorem 2.3]{hugel2004direct}, $\varinjlim \Ccal^{<\omega}$ is a class in a Tor-pair, so in particular it is also determined locally. Thus it suffices to show that for $M \in \Ccal$, $M_\m \in \varinjlim \Ccal^{<\omega}$ for each $\m \in \mspec R$. 

    By the assumption, $\Ccal_\m = \varinjlim(\mod R_\m \cap \Ccal_\m)$. The assumption $\LF$ implies that $\mod R_\m \cap \Ccal \subseteq \varinjlim \Ccal^{<\omega}$, so the claim follows as $\varinjlim \Ccal^{<\omega}$ is closed under direct limits. 
\end{proof}

\subsection{The case of cotorsion pairs and Tor-pairs}\label{ss:defn-cot-tor-pair} In our cases of interest, the class $\Ccal$ will fit into the structure of a cotorsion pair or even a Tor-pair, which allows for a more classical and useful reformulation of conditions $\GL$ and $\FD$. Given a subcategory $\Ccal$ of $\Mod R$, let $$\Ccal^\perp = \{N \in \Mod R \mid \Ext_R^i(M,N) = 0 ~\forall M \in \Ccal ~\forall i>0\}.$$ and define ${}^\perp \Ccal$ analogously. A \textit{hereditary cotorsion pair} is a pair $(\Ccal,\Dcal)$ of subcategories of $\Mod R$ such that $\Dcal = \Ccal^\perp$ and $\Ccal = {}^\perp \Dcal$. For a class of modules $\X$, we say that $(\Ccal,\Dcal)$ is generated by $\X$ if $\Dcal = \X^\perp$, and analogously it is cogenerated by $\X$ if $\Ccal = {}^\perp\X$.
For any hereditary cotorsion pair $(\Ccal, \Dcal)$, $\Ccal$ is closed under direct summands and transfinite extensions and thus the property $\FD$ gives a description of $\Ccal$ as precisely the direct summands of transfinite extensions of modules from $\Ccal^{<\omega}$. In fact, the theorem of Eklof and Trlifaj \cite[Corollary 6.14]{GT12} shows that $\Ccal$ satisfies $\FD$ if and only if $\Dcal = (\Ccal^{<\omega})^\perp$, that is, the hereditary cotorsion pair $(\Ccal,\Dcal)$ is generated by finitely generated modules.

Similarly, we let $$\Ccal^\top = \{N \in \Mod R \mid \Tor_i^R(M,N) = 0 ~\forall M \in \Ccal ~\forall i>0\}.$$ and define ${}^\top \Ccal$ analogously. A \textit{hereditary Tor-pair} is a pair $(\Ccal,\Dcal)^\top$ of subcategories of $\Mod R$ such that $\Dcal = \Ccal^\top$ and $\Ccal = {}^\top \Dcal$. In this case, $\Ccal$ is closed under direct limits and so if $\Ccal$ satisfies $\GL$, then we have a description of $\Ccal$ as precisely the direct limits of modules from $\Ccal^{<\omega}$. Any hereditary Tor-pair gives rise to a hereditary cotorsion pair $(\Ccal,\Ecal)$ cogenerated by character duals of objects in $\Dcal$, see \cite[Lemma 2.16(b), Lemma 5.17]{GT12}. By \cite[Theorem 2.3]{hugel2004direct}, $\Ccal$ satisfies $\GL$ if and only if $\Dcal = (\Ccal^{<\omega})^\top$, that is, the hereditary Tor-pair $(\Ccal,\Dcal)^\top$ is generated by finitely generated modules.

\subsection{Approximations}
Suppose $\Ccal$ is an isomorphism-closed collection of $R$-modules. A \textit{$\Ccal$-preenvelope} of an $R$-module $M$ is a homomorphism $\eta \colon M \to C$ with $C \in \Ccal$ such that $\Hom_R(\eta,C') \colon\Hom_R(C,C') \to \Hom_R(M, C')$ is a surjection.

We say that an isomorphism-closed collection of (finitely generated) $R$-modules $\Ccal$ is \textit{preenveloping} in $\Mod R$ ($\mod R$) if every (finitely generated) $R$-module has a $\Ccal$-preenvelope. In general, there is not always a correspondence between $\Ccal$ being preenveloping in $\Mod R$ and $\Ccal \cap \mod R$ being preenveloping in $\mod R$. However, if $\Ccal \subseteq \mod R$, then $\Ccal$ is preenveloping in $\mod R$ if and only if $\varinjlim \Ccal\subseteq \Mod R$ is preenveloping in $\Mod R$ if and only if $\varinjlim \Ccal$ is closed under products in $\Mod R$, \cite[4.2]{boev} \cite[3.11]{Kra01}. In particular, as mentioned above, $\varinjlim \Ccal$ satisfies $\GL$, \cite{Lenzing83}. 

Precovers are defined dually, however we will not refer to them here. 

\subsection{Homological dimensions} Given $n \geq 0$, we let 
$$\Pcal_n(R) = \{M \in \Mod R \mid \pd_R(M) \leq n\},$$ 
$$\Fcal_n(R) = \{M \in \Mod R \mid \fd_R(M) \leq n\}\text{, and}$$
$$\Ical_n(R) = \{M \in \Mod R \mid \id_R(M) \leq n\}$$ 
denote the subcategories of $\Mod R$ consisting of all modules whose projective, flat, or injective dimension, respectively, is at most $n$. If the ring $R$ is clear from context, we will write simply $\Pcal_n$, $\Fcal_n$, and $\Ical_n$. The same convention will be used for the further subcategories determined by dimensions we consider as well. Furthermore, we will write 
$$\Pcal_{<\infty} = \bigcup_{n \geq 0}\Pcal_n,$$
$$\Fcal_{<\infty} = \bigcup_{n \geq 0}\Fcal_n\text{, and}$$
$$\Ical_{<\infty} = \bigcup_{n \geq 0}\Ical_n$$
for the subcategories of all modules of finite projective, flat, or injective dimension, respectively. 

The subcategory $\Pcal_0$ of all projective $R$-modules is well known to satisfy $\FD$, as every projective module is a direct summand in a free one. The finite type of $\Pcal_1$ is was studied extensively in \cite{BH09} for various rings. For commutative noetherian rings, it was shown recently in \cite{hrbek2023finite} that $\Pcal_n$ satisfies $\FD$ if and only if $R$ satisfies Serre's condition $(S_n)$. In particular, $\Pcal_n$ satisfies $\FD$ for all $n \geq 0$ if and only if $R$ is Cohen--Macaulay.

The subcategory $\Fcal_0$ of all flat $R$-modules is well known to satisfy $\GL$ by the classical theorem proved independently by Govorov and Lazard. The validity of $\GL$ for $\Fcal_n$ was also studied in \cite{hrbek2023finite}. It turns out that $\Fcal_n$ satisfies $\GL$ if and only if $R$ satisfies ``almost'' Serre's condition $(C_{n+1})$. In particular, $\Fcal_n$ satisfies $\GL$ for all $n \geq 0$ if and only if $R$ is \textit{almost Cohen--Macaulay}, a generalization of Cohen--Macaulay rings defined in \cite{CFF02} as rings $R$ such that 
$$\depth(R_\p) \geq \dim(R_\p) - 1 \text{ for all } \p \in \spec R.$$

We remark the following corollary of independent interest.

\begin{cor}
The following are equivalent for a commutative noetherian ring $R$:
\begin{enumerate}
    \item[(i)] $R$ is almost Cohen--Macaulay,
    \item[(ii)] For any $k \geq 0$, $\Pcal_k$ satisfies $\LF$.
\end{enumerate}
\end{cor}
\begin{proof}
    This implication $(i) \implies (ii)$ follows by applying \cref{lift-GL} to $\Ccal = \Fcal_k$. Indeed, $\Fcal_k$ satisfies $\GL$ by the above reference, and we have $\Pcal_k^{<\omega}(R_\p) = \Fcal_k^{<\omega}(R_\p) \subseteq \Fcal_k(R)$.

    For the converse, assume that $\m$ is a maximal ideal such that $R_\m$ is not almost Cohen--Macaulay. If $(ii)$ holds for $R$ then it clearly also holds for $R_\m$, so it suffices to show that $(ii)$ fails for a local ring $R$ which is not almost Cohen--Macaulay. By the definition, $\depth(R)<\dim(R)-1$. On the other hand, \cite[Proposition 5.1, Corollary 5.3]{Bas62}, there is $\p \in \spec R$ with $\depth(R_\p) = \dim(R) - 1$. By the Auslander-Buchsbaum formula \cite{AB56}, there is a finitely generated $R_\p$-module $M$ such that $\pd_{R_\p}(M)=\depth(R_\p)$. Towards contradiction, let $N \in \Pcal_{\depth(R_\p)}^{<\omega}(R)$ such that $M$ is a direct summand of $N_\p$. Then $\pd_R(N) \leq \depth(R) < \depth(R_\p)$, resulting in $\pd_{R_\p}(M)<\depth(R_\p)$, a contradiction.
\end{proof}

\subsection{Gorenstein projective and flat modules}
For $n \geq 0$ we let 
$$\Gcal\Pcal_n(R) = \{M \in \Mod R \mid \Gpd_R(M) \leq n\}\text{ and}$$ 
$$\Gcal\Fcal_n(R) = \{M \in \Mod R \mid \Gfd_R(M) \leq n\}$$
denote the subcategories of $\Mod R$ consisting of modules of Gorenstein projective and Gorenstein flat dimension at most $n$. That $\GL$ holds for $\Gcal\Fcal_0$ for (finite-dimensional) Gorenstein rings is well-known \cite[Theorem 10.3.8]{EJ11}, however this can fail to hold outside of Gorenstein rings \cite{HJ11}. In fact, there are local artinian such examples \cite[Remark 2.9]{HJ11}, and for these also $\Gcal\Pcal_0$ fails $\FD$, as we now show.

\begin{lem}\label{GP0fail}
    Let $R$ be of Krull dimension $d<\infty$. If $\Gcal\Fcal_d$ fails $\GL$ then $\Gcal\Pcal_d$ fails $\FD$.
\end{lem}
\begin{proof}
    First, $R$ being of Krull dimension $d$ implies $\Gcal\Fcal_d = \Gcal\Pcal_d$, see \cite[Proposition 3.1 and Theorem 3.4]{Esmkhani2007}. Let $P \in \Gcal\Pcal_d$ be such that $P$ is not in $\varinjlim \Gcal\Pcal_d^{<\omega}$. 
    By the proof of \cite[Theorem 2.3]{hugel2004direct}, the (direct summands of the) transfinite extensions of modules from $\Gcal\Pcal_d^{<\omega}$ are contained in $\varinjlim \Gcal\Pcal_d^{<\omega}$, which concludes the proof as $\varinjlim \Gcal\Pcal_d^{<\omega} \subseteq \Gcal\Fcal_d$. 
    Alternatively, we give an explicit proof here. Since $\varinjlim \Gcal\Pcal_d^{<\omega}$ is closed under direct summands as noted in \cite[Lemma 1.2]{hugel2004direct}, we can towards contradiction assume that $P$ is a transfinite extension of modules from $\Gcal\Pcal_d^{<\omega}$. By applying the Hill Lemma, see the condition (H4) of \cite[Theorem 7.10]{GT12}, we obtain that any finitely generated submodule $M$ of $P$ is included in another submodule $N$ of $P$ such that $N \in \Gcal\Pcal_d^{<\omega}$. This would immediately yield $P \in \varinjlim \Gcal\Pcal_d^{<\omega}$, a contradiction.
\end{proof}
\begin{cor}\label{GP0artinian}
    Let $R$ be local, artinian, not Gorenstein, and such that $\Gcal\Pcal_0^{<\omega} \not\subseteq \Pcal_0$. Then $\Gcal\Pcal_0$ does not satisfy $\FD$.
\end{cor}
\begin{proof}
    Use \cref{GP0fail} and \cite[Remark 2.9]{HJ11}.
\end{proof}

\begin{quest}
    If $\Gcal\Pcal_n$ satifies $\FD$, does $\Gcal\Fcal_n$ necessarily satisfy $\GL$? \cref{GP0fail} and \cite[Proposition 3.1 and Theorem 3.4]{Esmkhani2007} show this is the case for $n \geq d$ when $R$ is of Krull dimension $d$. A positive answer would mirror the analogous result for the classes $\Pcal_n$ and $\Fcal_n$, see \cite{hrbek2023finite}.
\end{quest}

\section{Restricted homological dimensions}\label{S:restricted}
Following \cite{CFF02}, the \textit{(large) restricted projective dimension} and the \textit{(large) restricted flat dimension} of an $R$-module $M$ are defined as follows, in terms of Ext (resp., Tor) vanishing restricted to modules which are of finite injective (resp., flat) dimension:
$$\Rpd_R(M) = \sup \{n \geq 0 \mid \Ext_R^n(M,I) \neq 0 ~\exists I \in \Ical_{<\infty}\}\text{ and }$$
$$\Rfd_R(M) = \sup \{n \geq 0 \mid \Tor_n^R(M,F) \neq 0 ~\exists F \in \Fcal_{<\infty}\}.$$
Setting a notation similar to above, for any $n \geq 0$ we let
$$\Rcal\Pcal_n(R) = \{M \in \Mod R \mid \Rpd_R(M) \leq n\}\text{ and }$$ 
$$\Rcal\Fcal_n(R) = \{M \in \Mod R \mid \Rfd_R(M) \leq n\}$$
denote the subcategories of $\Mod R$ consisting of modules of restricted projective and restricted flat dimension at most $n$. The restricted projective and flat dimension generalize the Gorenstein projective and flat dimensions over Gorenstein rings in the following sense: If $R$ is a Gorenstein ring then $\Rpd_R = \Gpd_R$ and $\Rfd_R = \Gfd_R$. This is standard if $\dim(R)<\infty$, we clarify the infinite dimensional case in \cref{gpd-rpd}. The modules of large restricted flat dimension were also studied by Xu under the name of strongly torsion free modules, see \cite[Definition 5.4.2]{xu}.

We remark some further basic properties of these dimensions and subcategories.

\begin{prop}\label{rpdrfd-basics}
For any $n \geq 0$, the following hold:
    \begin{enumerate}
        \item $\Pcal_n \subseteq \Rcal\Pcal_n$ and $\Fcal_n \subseteq \Rcal\Fcal_n$,
        \item $\Rcal\Pcal_n \subseteq \Rcal\Fcal_n$ and $\Rcal\Pcal_n^{<\omega} = \Rcal\Fcal_n^{<\omega}$,
        \item $\Rcal\Pcal_n$ fits as a left-hand class in a hereditary cotorsion pair and $\Rcal\Fcal_n$ fits in a hereditary Tor-pair,
        \item $\Rpd_R M < \infty$ for any $M \in \mod R$.
        \item $\Rfd_R(M) = \sup\{\Rfd_{R_\m}(M_\m) \mid \m \in \mspec\} = \sup\{\Rfd_{R_\p}(M_\p) \mid \p \in \spec{R}\}$.
    \end{enumerate}
\end{prop}
\begin{proof}
    (1) and (3) are clear from the definition, (2) is \cite[Lemma 5.16]{CFF02}, and (4) is \cite[Theorem 1.1]{AIL10}.

    The locality condition (5) is clear due to the isomorphism $\Tor_R^i(-,-)_\p \cong \Tor_{R_\p}^i(-_{\p},-_{\p})_\p$ for any maximal ideal $\p$ and that $\fd_R(N) \geq \fd_R(N_\p) = \fd_{R_\p}(N_\p)$ holds for any $R$-module $N$.
\end{proof}
We will be chiefly interested in studying the condition $\FD$ for the subcategory $\Rcal\Pcal_n$ and $\GL$ for the subcategory $\Rcal\Fcal_\n$ for $n \geq 0$. Note that by \cref{rpdrfd-basics}, the latter condition is equivalent to $\Rcal\Fcal_n = \varinjlim \Rcal\Pcal_n^{<\omega}$.

In the generality of a commutative noetherian ring of finite Krull dimension, we establish $\FD$ for $\Rcal\Pcal_0$ and $\LF$ for $\Rcal\Fcal_0$ using a somewhat elementary argument. Let $\Lcal\Fcal_{<\infty} = \{M \in \Mod R \mid M_\p \in \Fcal_{<\infty}(R_\p) ~\forall \p \in \spec R\}$ denote the class of modules which are \textit{locally of finite flat dimension}. Note that, equivalently, $\Lcal\Fcal_{<\infty} = \{M \in \Mod R \mid M_\p \in \Fcal_{<\infty} ~\forall \p \in \spec R\}$, as $R$-flat and $R_\p$-flat dimension of $R_\p$-modules coincides. In particular, $\Lcal\Fcal_{<\infty} =\Fcal_{<\infty}$ in case $\dim(R)<\infty$.
Later we shall also need to consider the class $\Lcal\Ical_{<\infty} = \{M \in \Mod R \mid M_\p \in \Ical_{<\infty}(R_\p) ~\forall \p \in \spec R\}$ consisting of all modules which are \textit{locally of finite injective dimension}, that is, those $R$-modules $M$ such that $M_\p$ is of finite injective dimension over $R_\p$ for all primes $\p$.

\begin{lem}\label{gor}  Let $R$ be a commutative noetherian ring. Then there is a hereditary Tor-pair $(\Rcal\Fcal_0,\Lcal\Fcal_{<\infty})^\top$.

Furthermore, if $\dim(R)<\infty$ then the following also hold:
\begin{enumerate}
    \item[(i)]  $\Lcal\Fcal_{<\infty}=(\Rcal\Pcal_0^{<\omega})^{\top}$, and thus $\Rcal\Fcal_0$ satisfies $\GL$.
    \item[(ii)] There is a hereditary cotorsion pair $(\Rcal\Pcal_0,\Ical_{<\infty})$. Furthermore, $\Ical_{<\infty}=(\Rcal\Pcal_0)^{\perp}=(\Rcal\Pcal_0^{<\omega})^{\perp}$ and thus $\Rcal\Pcal_0$ satisfies $\FD$.
\end{enumerate}
\end{lem}

\begin{proof}
    Assume first $d = \dim(R)<\infty$, then $\Ical_{<\infty} = \Ical_d$ and $\Fcal_{<\infty} = \Fcal_d$, so that there is a hereditary cotorsion pair $(\Rcal\Pcal_0,\Ical_{<\infty})$ and a hereditary Tor-pair $(\Rcal\Fcal_0,\Fcal_{<\infty})^\top$. Let $\Scal$ denote the set of all $d$-th syzygies of all cyclic $R$-modules. By the Baer criterion, we have $\Scal^{\perp} = \Ical_d$, and by the dual Baer criterion, we have $\Scal^{\top} = \Fcal_d$. Since $\Rpd_R(M) \leq d$ for any $R$-module $M$, we have $\Scal \subseteq \Rcal\Pcal_0 \subseteq \Rcal\Fcal_0$, \cite[Lemma 5.16]{CFF02}. We have shown that $\Scal$ generates the cotorsion pair $(\Rcal\Pcal_0,\Ical_{<\infty})$ and the Tor-pair $(\Rcal\Fcal_0,\Fcal_{<\infty})^\top$, and as $\Scal$ consists of finitely generated modules, we are done, see \cref{ss:defn-cot-tor-pair}.

    It remains to show that there is a hereditary Tor-pair $(\Rcal\Fcal_0,\Lcal\Fcal_{<\infty})^\top$ without assuming $\dim(R) < \infty$. For $M \in \Rcal\Fcal_0$ we have $M_\p \in \Rcal\Fcal_0$ and $M_\p \in \Rcal\Fcal_0(R_\p)$, and so $\Tor_i^R(M,P)_\p \cong \Tor_i^{R_\p}(M_\p,P_\p)$ vanishes for all $i>0$ if and only if $P_\p \in \Fcal_{<\infty}(R_\p)$. It follows that $\Tor_i^R(M,P) = 0$ for all $i>0$ if and only if $P \in \Lcal\Fcal_{<\infty}$ and so $\Lcal\Fcal_{<\infty} = (\Rcal\Fcal_0)^{\top}$.
\end{proof}

\begin{rem}\label{CM-Rpd-rmk}
Let $\cm = \cm(R)$ denote the class of all finitely generated maximal Cohen--Macaulay modules. We remark that if $R$ is Cohen--Macaulay then $\cm = \Rcal\Pcal_0^{<\omega}$, this follows from \cite[Theorem 2.4(b)]{CFF02} or from \cite[Theorem 5.22]{CFF02}. As a particular consequence of \cref{gor}, for any Cohen--Macaulay ring $R$ of finite Krull dimension, the subcategory $\varinjlim \cm(R)$ fits into a hereditary cotorsion pair $(\varinjlim \cm(R),\Ical_{<\infty})$, a fact previously known only in the presence of a dualizing module, see \cite[Corollary 5.7]{Bird2020}.
\end{rem}

\begin{cor}\label{cor:rf0lift}
    If $R$ is a commutative noetherian ring of finite Krull dimension then $\Rcal\Fcal_0$ satisfies $\LF$.
\end{cor}
\begin{proof}
    Combine \cref{gor}(i) and \cref{lift-GL}.
\end{proof}
It is natural to ask if \cref{cor:rf0lift} holds for rings of infinite Krull dimension.
\begin{quest}
Does $\Rcal\Fcal_0$ satisfy $\GL$ over commutative noetherian rings of infinite Krull dimension?
\end{quest}
We are only able to show this in what follows in the presence of a pointwise dualizing module using approximation techniques. 
\subsection{Rings with a pointwise dualizing module}\label{SS:dualizing}
For the definition and references about pointwise dualizing modules, please see \cref{S:CMrings}. For the application in the next section, it suffices to recall that if $R$ is a Gorenstein ring then $R$ itself is a pointwise dualizing module. Note that if $R$ has a pointwise dualizing module then $R$ is Cohen--Macaulay. In particular, $\cm = \Rcal\Pcal_0^{<\omega}$ by \cref{CM-Rpd-rmk} and $\cm(R_\p) = \Rcal\Fcal_0 \cap \mod R_\p$ as $\Rcal\Fcal_0(R_\p) = (\Rcal\Fcal_0)_\p$.

\begin{lem}\label{MCMloc}
    Let $R$ be a Cohen--Macaulay commutative noetherian ring with a pointwise dualizing module $\Omega$. Then $\Rcal\Fcal_0$ satisfies $\LF$.
\end{lem}
\begin{proof}
    Fix $L \in \Rcal\Fcal_0 \cap \mod R_\p = \cm(R_\p)$ and let $N \in \mod R$ be such that $N_\p \cong L$. By \cite[Theorem 1.2]{miyachi1998cohen} (and its proof), there is a short exact sequence $0 \to K \to G \to N \to 0$ with $G \in \cm(R)$ and $K$ admitting a finite resolution by finite coproducts of copies of $\Omega$. Consider the localised sequence $0 \to K_\p \to G_\p \to N_\p \to 0$. Then $G_\p \in \cm(R_\p)$ (\cite[Theorem 2.13]{Bruns/Herzog:1998}) and $N_\p = L \in \cm(R_\p)$. On the other hand, $K_\p$ admits a finite resolution by the dualizing module $\Omega_\p$ over the local ring $R_\p$, and thus $K_\p \in \Ical_{<\infty}(R_\p)$. It follows that the latter short exact sequence splits, rendering $L$ a direct summand of $G_\p$, as desired.
\end{proof}

\begin{cor}\label{RF0-GL}
    Let $R$ be a Cohen--Macaulay commutative noetherian ring with a pointwise dualizing module $\Omega$. Then $\Rcal\Fcal_0$ satisfies $\GL$. 
\end{cor}
\begin{proof}
    Follows from \cref{LF-to-GL}, \cref{gor}, and \cref{MCMloc}.
\end{proof}

\begin{lem}\label{rpd-cotorpair}
Let $R$ be a Cohen--Macaulay commutative noetherian ring with a pointwise dualizing module $\Omega$. Then $\Lcal\Ical_{<\infty}=\cm^{\perp}$.
\end{lem}
\begin{proof}
    Note that $\Ext_R^i(G,M)_\p \cong \Ext_{R_\p}^i(G_\p,M_\p)$ for any $G \in \cm$ and $\p \in \spec R$. This yields $\Lcal\Ical_{<\infty} \subseteq \cm^{\perp}$, as if $G \in \cm$ then $G_\p \in \cm(R_\p) \subseteq {}^{\perp}\Ical_{<\infty}(R_\p)$ by \cite[Corollary 5.23]{CFF02}. 
    
    The converse inclusion follows directly by applying \cref{MCMloc}. Indeed, let $M \in \cm^{\perp}$, $\p \in \spec R$, and $L \in \cm(R_\p) = \Rcal\Pcal_0(R_\p)^{<\omega}$. Then $\Ext_{R_\p}^1(L,M_\p) \cong \Ext_R^1(G,M)_\p = 0$ for some $G \in \cm(R)$, and thus $M_\p$ is of finite injective dimension as an $R_\p$-module by \cref{gor}.
\end{proof}

\begin{rem}
    We will show that also $\Rcal\Pcal_0$ satisfies $\FD$ in the setting of a ring with pointwise dualizing module in \cref{CMproj} by applying \cref{RfdCMfd}.
\end{rem}

\section{Gorenstein rings}\label{S:Gor}
In this section, we shall prove that $\Gcal\Pcal_n$ satisfies $\FD$ and $\Gcal\Fcal_n$ satisfies $\GL$ for any $n\geq 0$ in case $R$ is a Gorenstein ring which is not necessarily of finite Krull dimension.

\begin{lem}\label{gorinfdim}
    Let $R$ be Gorenstein. Then there is a hereditary cotorsion pair of the form $(\Gcal\Pcal_0,\Lcal\Ical_{<\infty})$ generated by $\Gcal\Pcal_0^{<\omega}$. In particular, $\Gcal\Pcal_0$ satisfies $\FD$.
\end{lem}
\begin{proof}
    By \cref{rpd-cotorpair}, we have a hereditary cotorsion pair $(\mathcal{E},\Lcal\Ical_{<\infty})$ generated by $\cm$, our goal is to show that it coincides with the hereditary cotorsion pair $(\Gcal\Pcal_0,\mathcal{W})$. As $\cm = \Gcal\Pcal_0^{<\omega}$, we have $\mathcal{E} \subseteq \Gcal\Pcal_0$. By \cite[Lemma 1.6]{hugel2004direct}, it is enough to show that $\Gcal\Pcal_0 \cap \Lcal\Ical_{<\infty} \subseteq \mathcal{E} \subseteq \Lcal\Ical_{<\infty}$. Let $M \in \Gcal\Pcal_0 \cap \Lcal\Ical_{<\infty}$, then for any $\p \in \spec R$ the $R_\p$-module $M_\p$ is of finite flat dimension and Gorenstein flat. It follows that $M_\p$ is flat and thus $M$ is a flat $R$-module. By \cite[Remark A.10]{shaul2023acyclic} it follows that $M$ is projective and thus $M \in \mathcal{E}$, as desired.
\end{proof}

We remark the following corollary, which is well-known in case of finite Krull dimension. 

\begin{cor}\label{gpd-rpd} Let $R$ be a Gorenstein ring. Then $\Gpd_R = \Rpd_R$ and $\Gfd_R = \Rfd_R$.
\end{cor}
\begin{proof}
    It suffices to show $\Rcal\Pcal_0 = \Gcal\Pcal_0$ and $\Rcal\Fcal_0 = \Gcal\Fcal_0$, note that the inclusions $\Rcal\Pcal_0 \supseteq \Gcal\Pcal_0$ and $\Rcal\Fcal_0 \supseteq \Gcal\Fcal_0$ hold in general. Since $R$ is Gorenstein we have $\Rcal\Pcal_0^{<\omega} = \Gcal\Pcal_0^{<\omega}$ \cite[Theorem 3.19]{holm} and \cite[Lemma 5.16]{CFF02}. By \cref{RF0-GL}, we have $\Rcal\Fcal_0 = \varinjlim \Rcal\Pcal_0^{<\omega} = \varinjlim \Gcal\Pcal_0^{<\omega} \subseteq \Gcal\Fcal_0$. Similarly, by \cref{gorinfdim} we have that any $P \in \Rcal\Pcal_0$ is a direct summand in a transfinite extension of modules from $\Rcal\Pcal_0^{<\omega} =  \Gcal\Pcal_0^{<\omega}$, and thus $M \in \Gcal\Pcal_0$.
\end{proof}

Now we are ready to prove the two promised results.

\begin{thm}\label{Gpd-deconstr} If $R$ is Gorenstein, then $\Gcal\Pcal_{n}$ satisfies $\FD$ for all $n\geq 0$.  
\end{thm}

\begin{proof} The case of $n=0$ is \cref{gorinfdim}, we extend this to the case of $n>0$. If $M\in \Gcal\Pcal_{n}$, then there is an exact sequence $0\to M\to P\to N\to 0$, where $N\in \Gcal\Pcal_0$ and $P\in \Pcal_n$, see \cite[Lemma 2.17]{func}. Taking the pull-back square of this and $0\to \syz N\to Q\to N\to 0$ for some $Q\in \Pcal_0$, we obtain an exact sequence $0\to \syz N\to M\oplus Q\to P\to 0$. Since $\syz N\in \Gcal\Pcal_0$, the claim follows by the $n=0$ case and \cite[Theorem A]{hrbek2023finite}. Indeed, consider the hereditary cotorsion pair $(\Ccal,\Dcal)$ generated by $\Gcal\Pcal_n^{<\omega}$. Then $\Ccal$ contains $\syz N$, as $\Gcal\Pcal_0^{<\omega} \subseteq \Ccal$. Also, $\Ccal$ contains $P$, as $\Pcal_n^{<\omega} \subseteq \Ccal$ and using \cite[Theorem A]{hrbek2023finite}. Finally, $\Ccal$ is closed under extensions and direct summands, which shows that $M \in \Ccal$, as desired.
\end{proof}

\begin{thm}\label{Gfd-deconstr}
    If $R$ is Gorenstein then $\Gcal\Fcal_{n}$ satisfies $\GL$ for all $n 
    \geq 0$.
\end{thm}
\begin{proof}
    The case of $n=0$ follows directly from \cref{RF0-GL} and \cref{gpd-rpd}. Let $n>0$ and $M \in \Gcal\Fcal_{n}$, then by \cite[Lemma 2.19]{func}, there is an exact sequence $0 \to M \to H \to A \to 0$, where $A \in \Gcal\Fcal_{0}$ and $H \in \Fcal_{n}$. By \cite[Corollary 2.4]{hugel2004direct}, the direct limit closure of $\Gcal\Fcal_{n}^{<\omega}$ fits as a left-hand class $\Ccal$ in a hereditary cotorsion pair, and as such, is closed under kernels of epimorphisms. We know that $\Fcal_{n}$ satisfies $\GL$ by \cite[Theorem B]{hrbek2023finite}, and thus $H \in \Ccal$. Since also $A \in \Ccal$ by the $n=0$ case, we conclude that $M \in \Ccal$.
\end{proof}

For certain artinian rings, we also get the converse statements.

\begin{cor}\label{R-Gor-equiv-cond}
    Let $R$ be a local artinian ring such that $\Gcal\Pcal_0^{<\omega} \not\subseteq \Pcal_0$. Then the following are equivalent:
    \begin{enumerate}
        \item[(i)] $R$ is Gorenstein,
        \item[(ii)] $\Gcal\Pcal_0$ satisfies $\FD$,
        \item[(iii)] $\Gcal\Fcal_0$ satisfies $\GL$.
        \item[(iv)] $\Gcal\Pcal_n$ satisfies $\FD$ for any $n \geq 0$,
        \item[(v)] $\Gcal\Fcal_n$ satisfies $\GL$ for any $n \geq 0$.
    \end{enumerate}
\end{cor}
\begin{proof}
    $(i) \implies (ii):$ This is \cref{gor} together with \cref{gpd-rpd}.

    $(ii) \implies (iii):$ \cref{GP0artinian}.

    $(iii) \implies (i)$: \cite[Remark 2.9]{HJ11}.

    $(i) \implies (iv),(v):$ \cref{Gpd-deconstr} and \cref{Gfd-deconstr}.

    $(iv) \implies (ii)$ and $(v) \implies (iii)$ are trivial.
\end{proof}

\begin{rem} Let $A$ be an artinian non-Gorenstein local ring. Then, $R=A[[T]]/(T^2)$ is also an artinian non-Gorenstein local ring over which $R/(T)\in \Gcal\Pcal_0^{<\omega} \setminus \Pcal_0$, so  $\Gcal\Fcal_0$ fails $\GL$ by  \cref{R-Gor-equiv-cond}. This example also shows that ``countably presented'' in \cite[Theorem 10.2]{resol} cannot always be improved to ``finitely presented''. 
\end{rem}

\begin{quest}\label{smallGP=P-implies-big}
    Let $R$ be an artinian ring which is not Gorenstein and such that $\Gcal\Pcal_0^{<\omega} \subseteq \Pcal_0$ (see \cite{Tak08}, \cite{Chen17}). Does it hold that $\Gcal\Pcal_0 \subseteq \Pcal_0$?
\end{quest}

\begin{rem}
    A positive answer to \cref{smallGP=P-implies-big} would mean that the condition $\Gcal\Pcal_0^{<\omega} \not\subseteq \Pcal_0$ in \cref{R-Gor-equiv-cond} could be omitted. Explicitly, this would mean that over an artinian ring, conditions $(ii)-(v)$ of \cref{R-Gor-equiv-cond} are equivalent to the condition 
    \begin{enumerate}
        \item[(i)']$R$ is Gorenstein or $\Gcal\Pcal_0^{<\omega} \subseteq \Pcal_0$.
    \end{enumerate}This would give another characterisation of the commutative artinian rings which are virtually Gorenstein, see \cite[Theorem 5]{Beligiannis-thick-2008}.
\end{rem}

\section{Cohen--Macaulay rings with pointwise dualizing modules}\label{S:CMrings}
Now we extend the scope to rings admitting a pointwise dualizing module (which are automatically Cohen--Macaulay). As we observed in \cref{GP0fail}, the Gorenstein projectives can fail to satisfy $\FD$ even over artinian rings. Instead, we consider another kind of homological dimension introduced in \cite{HJ07}. A \textit{semidualizing} module is a finitely generated module $C$ such that the natural homothety map $R \to \Hom_R(C,C)$ is an isomorphism and $\Ext_R^i(C,C)=0$ for all $i>0$. A semidualizing module $\Omega$ is \textit{pointwise dualizing} if it belongs to $\Lcal\Ical_{<\infty}$ and \textit{dualizing} if it belongs to $\Ical_{<\infty}$.

Given an $R$-module $M$, the \textit{trivial extension} of $R$ is the commutative $R$-algebra $R \ltimes M$ whose underlying module is $R \oplus M$ and the multiplication is given by the rule $(r,m)\cdot(s,n) = (rs,rn+sm)$. Note that $R$ is a ring quotient of $R \ltimes M$ over the ideal $0 \oplus M$, and so every $R$-module is also naturally an $R \ltimes M$-module. The Cohen--Macaulay projective and flat dimensions of a module $M$ are then defined as:
$$\CMpd_R(M) = \inf \{\Gpd_{R \ltimes C}(M) \mid C \text{ a semidualizing $R$-module}\},$$
$$\CMfd_R(M) = \inf \{\Gfd_{R \ltimes C}(M) \mid C \text{ a semidualizing $R$-module}\}.$$
If $R$ has a dualizing module then for $M$ finitely generated $\CMpd_R(M) = \CMfd_R(M)$ is equal to the CM-dimension of Gerko \cite{G01}. 

As before, we denote $\Ccal\Mcal\Pcal_n = \{M \in \Mod R \mid \CMpd_R(M) \leq n\}$ and $\Ccal\Mcal\Fcal_n = \{M \in \Mod R \mid \CMfd_R(M) \leq n\}$.
\begin{lem}\label{pointwise-dualizing}
    Let $R$ be Cohen--Macaulay with a pointwise dualizing module $\Omega$. Then $R \ltimes \Omega$ is a Gorenstein ring.
\end{lem}
\begin{proof}
    This is \cite[Theorem 2.2]{J03} in the case that $R$ is local. The general case follows from the fact that $R$ being Gorenstein is a local property and a localization at $\p$ of a pointwise dualizing module over $R$ is a dualizing $R_\p$-module by definition. Moreover, the projection $\pi: R \ltimes \Omega \to R$ induces a bijection $\pi^*: \spec R \cong \spec(R \ltimes \Omega)$ under which $(R \ltimes \Omega)_{\pi^*(\p)} \cong R_\p \ltimes \Omega_\p$ for any $\p$.
\end{proof}
\begin{lem}\label{RfdCMfd}
    Let $R$ be Cohen--Macaulay with a pointwise dualizing module $\Omega$. Then $\Rfd_R = \CMfd_R = \Gfd_{R \ltimes \Omega}$ and $\Rpd_R = \CMpd_R = \Gpd_{R \ltimes \Omega}$.
\end{lem}
\begin{proof}
    It is easier to prove the first claim. For any maximal ideal $\m$ of $R$ we have $\Rfd_{R_\m} = \Gfd_{R_\m \ltimes \Omega_\m}$ by \cite[Lemma 4.14]{SSY20} and \cref{pointwise-dualizing}. In view of \cref{gpd-rpd} and \cref{rpdrfd-basics}(5), we see that $\Rfd_R = \Rfd_{R \ltimes \Omega} = \Gfd_{R \ltimes \Omega}$. By the definition, we have $\CMfd_R \leq \Gfd_{R \ltimes \Omega}$, so it remains to show the inequality $\Rfd_R \leq \CMfd_R$, or equivalently $\Rfd_R \leq \Gfd_{R \ltimes C}(M)$ for any $C$ semidualizing. This follows from \cite[Lemma 4.14]{SSY20} (note that the proof of \textit{loc. cit.} does not need to assume the ring $R$ to be local).

    By \cite[Lemma 2.12]{HJ06}, we obtain that $(\Ccal\Mcal\Pcal^{<\omega})^{\perp}$ is closed under kernels of epimorphisms. As a consequence, we get that $\Ext_R^i(P,I)=0$ for any $P \in \Ccal\Mcal\Pcal_0$ and $I \in \Ical_{<\infty}$. This yields a chain of inequalities $$\Rpd_R \leq \CMpd_R \leq \Gpd_{R \ltimes \Omega} = \Rpd_{R \ltimes \Omega},$$
    where the first inequality follows by the above argument, the second one is directly from the definition of $\CMpd_R$, and the last equality is \cref{gpd-rpd}. 
    
    It remains to show $\Rpd_{R \ltimes \Omega} \leq \Rpd_R$. For that, we show that any $M \in \Rcal\Pcal_0(R \ltimes \Omega)$ belongs to $\Rcal\Pcal_0(R)$. For any $I \in \Ical_{<\infty}$ we have derived adjunction isomorphisms
    $$\RHom_R(M,I) \cong \RHom_R(M \otimes_{R \ltimes \Omega} (R \ltimes \Omega), I) \cong \RHom_{R \ltimes \Omega}(M,\RHom_R(R \ltimes \Omega,I)).$$ 
    By \cite[Lemma 3.1]{HJ07}, $\RHom_R(R \ltimes \Omega,I)$ is isomorphic to a stalk of an $R \ltimes \Omega$-module of finite injective dimension, and so the positive cohomology of the latter $\RHom$ vanishes. This shows $\Ext_R^i(M,I)=0$ for any $i>0$ and $I \in \Ical_{<\infty}$, which in turn yields $M \in \Rcal\Pcal_0(R)$.
 \end{proof}

 Recall that over local Cohen--Macaulay rings, weak balanced big CM modules are exactly those for which $\Rfd=0$, see \cite[Definition 4.3, Proposition 2.4]{Hol17} and \cite[Corollary 3.3]{CFF02}. 
In view of this, we note that the second claim of the following result extends  \cite[Theorem A]{Hol17} from $n=0$ to arbitrary $n \geq 0$ and also to non-local rings.
\begin{thm}\label{CMproj}
    Let $R$ be Cohen--Macaulay with a pointwise dualizing module. For any $n\geq 0$, $\Ccal\Mcal\Pcal_n$ satisfies $\FD$ and $\Ccal\Mcal\Fcal_n$ satisfies $\GL$.
\end{thm}
\begin{proof}
    Let $M \in \Ccal\Mcal\Fcal_n$. Then $\Gfd_{R \ltimes \Omega}M \leq n$ and by the \cref{Gfd-deconstr} and \cref{pointwise-dualizing}, $M$ is a direct limit of finitely generated $(R \ltimes \Omega)$-modules $M_i$ of $\Gpd_{R \ltimes \Omega} \leq n$. Viewed over $R$, we have that $M$ is a direct limit of the same direct system of finitely generated $R$-modules $M_i$ and $M_i \in \Ccal\Mcal\Fcal_n^{<\omega}$ as $\CMfd_R(M_i) = \Gpd_{R \ltimes \Omega}(M_i)$.

    Let $M \in \Ccal\Mcal\Pcal_n$, then $\Gpd_{R \ltimes \Omega}(M) \leq n$ by the above. Using \cref{Gpd-deconstr} and \cref{pointwise-dualizing} that $M$ is a summand of an $R \ltimes \Omega$-module which is filtered by finitely generated $R \ltimes \Omega$-modules of $\Gpd_{R \ltimes \Omega} \leq n$. Any such module has $\CMpd_R \leq n$ over $R$. As retractions over $R \ltimes \Omega$ are retractions also over $R$, we are done.
\end{proof}

As a consequence, we can also generalize \cite[Theorem C]{Hol17} to non-local Cohen--Macaulay rings and more. 

\begin{cor}\label{envelop1} Let $R$ be Cohen--Macaulay with a pointwise dualizing module. Then, for every $n\ge 0$, $\Ccal\Mcal\Fcal_n\cap \mod R$ is preenveloping in $\mod R$.  
\end{cor}

\begin{proof} It follows from \Cref{RfdCMfd}, \cite[Corollary 3.3]{CFF02}
and \cite[Theorem 3.2.26]{EJ11}  that   $\Ccal\Mcal\Fcal_n$ is closed under products. Now \cite[Theorem (4.2)]{boev} and \Cref{CMproj} finishes the proof. 
\end{proof}

Note that \cref{CMproj} show together with \cref{RfdCMfd} that $\Rcal\Pcal_n$ satisfies $\FD$ and $\Rcal\Fcal_n$ satisfies $\GL$ for all $n \geq 0$ in case $R$ admits a pointwise dualizing module. Also, the case of $n=0$ holds true for any ring $R$ of finite Krull dimension by \cref{gor}. It is natural then to ask the following.
\begin{quest}
    Does $\Rcal\Pcal_n$ satisfy $\FD$ and does $\Rcal\Fcal_n$ satisfy $\GL$ for $n > 0$ over \textit{any} commutative noetherian ring $R$?
\end{quest}

In the next section, we provide more evidence by proving a positive answer to the second part of the question in the case of an almost Cohen--Macaulay ring of finite Krull dimension.

\section{Govorov--Lazard for restricted flat dimensions over almost Cohen--Macaulay rings}\label{S:ACM}
In this section, we propose an approach to establishing $\GL$ for the class $\Rcal\Fcal_n$ for any $n \geq 0$ based on the recent classification result \cite{HHLG24} for hereditary Tor-pairs cogenerated by modules of finite flat dimension. A hereditary Tor-pair $(\Ccal,\Dcal)^\top$ is cogenerated by a class $\Ecal \subseteq \Fcal_{<\infty}$ of modules of finite flat dimension if $\Ccal = {}^\top \Ecal$. We claim that this is equivalent to $\Dcal \subseteq \Lcal\Fcal_{<\infty}$. One inclusion is trivial, the other follows from the fact that ${}^\top \Ecal = {}^\top \bigcup_{\p \in \spec R}\Ecal_\p$, where $\Ecal = \{E_\p \mid E \in \Ecal\}$ and if $\Ecal \subseteq \Lcal\Fcal_{<\infty}$ we have $\Ecal_\p \subseteq \Fcal_{<\infty}(R)$. The mentioned classification can thus be formulated as follows.
\begin{thm}\label{HHLG-class}\cite[Theorem 4.17]{HHLG24}
    Let $R$ be a commutative noetherian ring. Then there is a bijection between:
    \begin{enumerate}
        \item[(i)] hereditary Tor-pairs $(\Ccal,\Dcal)^\top$ with $\Dcal \subseteq \Lcal\Fcal_{<\infty}$,
        \item[(ii)] functions $\ef: \spec R \to \ZZ_{\geq 0}$ such that $\ef(\p) \leq \depth R_\p$ for any $\p \in \spec R$.
    \end{enumerate}
    The bijections assigns to a function $\ef$ a Tor-pair $(\Ccal_\ef,\Dcal_\ef)^\top$, where 
    $$\Ccal_\ef = \{M \in \Mod R \mid \depth M_\p \geq \ef(\p) ~\forall \p \in \spec R\}.$$
\end{thm}

The following lemma shows for that for a finite dimensional ring, if a Tor-pair $(\Ccal,\Dcal)^\top$ is subject to the classification of \cref{HHLG-class}, then the Tor-pair generated by $\Ccal^{<\omega}$ is also subject to the same classification.

\begin{lem}\label{inducedTorpair}
    Let $R$ be a commutative noetherian ring. Let $(\Ccal,\Dcal)^\top$ be a hereditary Tor-pair with $\Dcal \subseteq \Lcal\Fcal_{<\infty}$. Then there is a hereditary Tor-pair $(\varinjlim\Ccal^{<\omega},\widetilde{\Dcal})^\top$. If $\dim(R)<\infty$ then $\widetilde{\Dcal} \subseteq \Lcal\Fcal_{<\infty}$ (=$\Fcal_{<\infty}$).
\end{lem}
\begin{proof}
    The existence of the hereditary Tor-pair $(\varinjlim\Ccal^{<\omega},\widetilde{\Dcal})^\top$ follows essentially from \cite[Corollary 2.4]{hugel2004direct}. Assume $\dim(R)<\infty$. By \cref{gor}, we have the hereditary Tor-pair $(\Rcal\Fcal_0,\Fcal_{<\infty})^\top$ and in addition, $\Rcal\Fcal_0 = \varinjlim \Rcal\Pcal_0^{<\omega}$. Since $\Rcal\Pcal_0^{<\omega} \subseteq \Ccal^{<\omega}$, we have $\Rcal\Fcal_0 \subseteq \varinjlim \Ccal^{<\omega}$, and thus $\widetilde{\Dcal} \subseteq \Fcal_{<\infty}$.
\end{proof}
As shown in \cite[\S 6]{HHLG24}, the hereditary Tor-pair $(\Rcal\Fcal_n,\Dcal_n)^\top$ is the subject to the classification above and corresponds to the function $\ef_n(\p) = \max(0,\depth R_\p -n)$ via \cref{HHLG-class}. More generally, given a hereditary Tor-pair $(\Ccal,\Dcal)^\top$ and $i>0$, there is the induced hereditary Tor-pair $(\Ccal_{(i)},\Dcal_{(i)})^\top$ cogenerated by $i$-th syzygies of objects from $\Dcal$, see \cite[\S 3.2]{HHLG24}. If $\ef: \spec R \to \ZZ_{\geq 0}$ is a function as in \cref{HHLG-class}, then \cite[Proposition 4.14]{HHLG24} shows that $(\Ccal_\ef)_{(i)} = \Ccal_{\max(0,\ef-i)}$. Thus, in this notation, $(\Rcal\Fcal_n)_{(i)} = \Rcal\Fcal_{n+i}$.

We are ready to prove the following criterion. 

\begin{thm}\label{criterionRFn}
    Let $R$ be a commutative noetherian ring and consider the following conditions:
    \begin{enumerate}
        \item[(i)] $\Rcal\Fcal_n$ satisfies $\GL$ for any $n \geq 0$,
        \item[(ii)] $\Rcal\Fcal_n$ satisfies $\LF$ for any $n \geq 0$,   
        \item[(iii)] for every $\p \in \spec R$ there is $M \in \Rcal\Pcal_{\depth R_\p}^{<\omega}$ such that $\p \in \ass(M)$.
    \end{enumerate}
    Then $(i) \iff (ii) \implies (iii)$. If $\dim(R)<\infty$ then also $(iii) \implies (i)$.
\end{thm}
\begin{proof}
    $(i) \implies (ii):$ This is \cref{lift-GL}.

    $(ii) \implies (iii)$: Since $k(\p)$ belongs to $\Rcal\Pcal_{\depth R_\p}^{<\omega}(R_\p)$ (follows directly from \cite[Theorem 2.4(b)]{CFF02}), $(ii)$ yields $M \in \Rcal\Pcal_{\depth R_\p}^{<\omega}(R)$ such that $M$ is a direct summand in $N_\p$. As $\ass(k(\p))=\{\p\}$, we clearly have $\p \in \ass(M)$.

    $(iii) \implies (i)$ under $d=\dim(R)<\infty$: We prove this by backward induction on $n = d, d-1, d-2, \ldots, 0$. The case of $n=d$ is vacuous as $\Rcal\Fcal_d = \Mod R$. By \cref{inducedTorpair} and \cref{HHLG-class}, there is a function $\ef: \spec R \to \ZZ_{\geq 0}$ bounded by the depth function such that $\varinjlim \Rcal\Pcal_n^{< \omega} = \{M \in \Mod R \mid \depth M_\p \geq \ef(\p) ~\forall \p \in \spec R\}$. It suffices to show that $\ef = \ef_n$. Since $\varinjlim \Rcal\Pcal_n^{<\omega} \subseteq \Rcal\Fcal_n$, we have $\ef \leq \ef_n$, so it remains to prove the other inequality.
    
    Consider first $\p \in \spec R$ such that $\ef_n(\p) > 0$, so $\depth R_\p >n$. By \cite[Proposition 4.14]{HHLG24}, $k(\p) \in \Rcal\Fcal_{n+\ef_n(\p)}$, and by the induction, we have $\varinjlim \Rcal\Pcal_{n+\ef_n(\p)}^{<\omega} = \Rcal\Fcal_{n+\ef_n(\p)} = (\varinjlim \Rcal\Pcal_n^{<\omega})_{(\ef_n(\p))}$, which immediately implies using \cite[Proposition 4.14]{HHLG24} again that $\ef(\p) \geq \ef_n(\p)$. 

    Finally, let us handle the case $\ef_n(\p) = 0$, which translates again to $k(\p) \in \Rcal\Fcal_n$. Then the finitely generated module $M$ of assumption $(ii)$ belongs to $\Rcal\Pcal_{\depth R_\p}^{<\omega} \subseteq \Rcal\Pcal_n^{<\omega}$, as $0 = \ef_n(\p) \geq \depth R_\p - n$. Since $\p \in \ass(M)$ we have $\depth M_\p = 0$, and thus $\ef(\p) = 0$.

    $(ii) \implies (i):$ 
    That $\Rcal\Fcal_n(R_\p)$ satisfies $\GL$ is shown by the implications $(ii)\implies (iii) \implies (i)$ above as $\dim (R_\p) < \infty$. Then one can apply \cref{LF-to-GL} as one can easily check that $\Rcal\Fcal_n(R_\p) = \Rcal\Fcal_n \cap \Mod {R_\p}$ and that $\Rcal\Fcal_n(R)$ is defined locally.
\end{proof}

Recall that $R$ is called \textit{almost Cohen--Macaulay} if $\dim(R_\p) - \depth R_\p \leq 1$ for all $\p \in \spec R$. Recall from \cite[Lemma 3.1]{CFF02} that $R$ is almost Cohen--Macaulay if and only if $\depth R_\p$ is equal to the grade of $\p$ for all $\p \in \spec R$.
\begin{cor}\label{thm-acm}
    Let $R$ be an almost Cohen--Macaulay ring of finite Krull dimension. Then the equivalent conditions of \cref{criterionRFn} are satisfied. In particular, $\Rcal\Fcal_n$ satisfies $\GL$ for any $n \geq 0$.
\end{cor}
\begin{proof}
    We check the condition $(iii)$ of \cref{criterionRFn}. Let $\p$ be a prime ideal of $R$ and $n = \depth R_\p$. Then the grade of $\p$ is also $n$ and so there is a regular sequence $(a_1,\ldots,a_n)$ of length $n$ contained in $\p$. Put $M = R/(a_1,\ldots,a_n)$. The local ring $M_\p = R_\p/(a_1,\ldots,a_n)_\p$ has depth zero and thus $\p \in \ass(M)$. It remains to show that $M \in \Rcal\Pcal_n$, but this is clear as $M \in \Pcal_n$.
\end{proof}

\begin{cor}\label{envelop2} Let $R$ be an almost Cohen--Macaulay ring of finite Krull dimension. Then, for every $n\ge 0$, $\Rcal\Fcal_n\cap \mod R$ is preenveloping in $\mod R$.  
\end{cor}

\begin{proof} It follows from  \cite[Corollary 3.3]{CFF02}
and \cite[Theorem 3.2.26]{EJ11}  that   $\Rcal\Fcal_n$ is closed under products. Now \cite[Theorem (4.2)]{boev} and \Cref{thm-acm} finishes the proof.  
\end{proof}

\begin{rem}
Condition (ii) of \cref{criterionRFn} requires to check the finite lifting property for each prime. However, collecting some results, we detail which primes it remains to check satisfy the finite lifting property. 

In view of \cref{thm-acm}, it is sufficient to only consider the primes for which either $R_\p$ is not almost Cohen--Macaulay, or $\p$ a maximal element of the almost Cohen--Macaulay locus of the spectrum. 

If additionally $R$ is of finite Krull dimension, the finite lifting property is known to  hold for those associated primes for which $\p$ is associated to $R$ and is in the almost Cohen--Macaulay locus of $R$. In this case, $\Mod R_\p = \Rcal\Fcal_0$, so the finite lifting property holds for these primes by \cref{cor:rf0lift}. 
\end{rem}

\begin{rem}
    \cref{thm-acm} in particular provides a generalization of the Govorov--Lazard part of \cref{CMproj} to a finite-dimensional Cohen--Macaulay ring $R$ in the absence of a dualizing module. Sahandi, Sharif, and Yassemi \cite{SSY20} provided a generalization of the notion of the CM-flat dimension $\CMfd$ for local rings even in the absence of a dualizing module and proved that it coincides with the restricted flat dimension for local Cohen--Macaulay rings \cite[Theorem 3.3, Corollary 4.2]{SSY20}. Thus, \cref{thm-acm} provides the Govorov--Lazard property for the class of modules of Sahandi-Sharif-Yassemi CM-flat dimension bounded by any integer for a local Cohen--Macaulay ring.
\end{rem}

\begin{quest}\label{Q:Rfd}
Do the conditions of \cref{criterionRFn} (equivalent to each other, at least, if $\dim(R)<\infty)$ hold true for any commutative noetherian ring $R$?
\end{quest}

\bibliographystyle{plain}
\bibliography{mainbib}

\begin{thebibliography}{10}

\bibitem{hugel2004direct}
Lidia Angeleri~H{\"u}gel and Jan Trlifaj.
\newblock Direct limits of modules of finite projective dimension.
\newblock {\em Rings, Modules, Algebras, and Abelian Groups, LNPAM}, 236:27--44, 2004.

\bibitem{AB56}
Maurice Auslander and David~A Buchsbaum.
\newblock Homological dimension in noetherian rings.
\newblock {\em Proceedings of the National Academy of Sciences}, 42(1):36--38, 1956.

\bibitem{AIL10}
Luchezar Avramov, Srikanth Iyengar, and Joseph Lipman.
\newblock Reflexivity and rigidity for complexes, i: Commutative rings.
\newblock {\em Algebra \& Number Theory}, 4(1):47--86, 2010.

\bibitem{Bas62}
Hyman Bass.
\newblock Injective dimension in noetherian rings.
\newblock {\em Transactions of the American Mathematical Society}, 102(1):18--29, 1962.

\bibitem{BH09}
Silvana Bazzoni and Dolors Herbera.
\newblock Cotorsion pairs generated by modules of bounded projective dimension.
\newblock {\em Israel Journal of Mathematics}, 174(1):119--160, 2009.

\bibitem{Beligiannis-thick-2008}
Apostolos Beligiannis and Henning Krause.
\newblock Thick subcategories and virtually {G}orenstein algebras.
\newblock {\em Illinois J. Math.}, 52(2):551--562, 2008.

\bibitem{Bird2020}
Isaac Bird.
\newblock Two definable subcategories of maximal {C}ohen--{M}acaulay modules.
\newblock {\em Journal of Pure and Applied Algebra}, 224(6):106250, 2020.

\bibitem{Bruns/Herzog:1998}
Winfried Bruns and J\"{u}rgen Herzog.
\newblock {\em Cohen--{M}acaulay rings}, volume~39 of {\em Cambridge Studies in Advanced Mathematics}.
\newblock Cambridge University Press, Cambridge, 1998.

\bibitem{Chen17}
Xiao-Wu Chen.
\newblock Gorenstein homological algebra of artin algebras.
\newblock {\em arXiv preprint arXiv:1712.04587}, 2017.

\bibitem{CFF02}
Lars~Winther Christensen, Hans-Bj{\o}rn Foxby, and Anders Frankild.
\newblock Restricted homological dimensions and {C}ohen--{M}acaulayness.
\newblock {\em Journal of Algebra}, 251(1):479--502, 2002.

\bibitem{func}
Lars~Winther Christensen, Anders Frankild, and Henrik Holm.
\newblock On {Gorenstein} projective, injective and flat dimensions -- a functorial description with applications.
\newblock {\em J. Algebra}, 302(1):231--279, 2006.

\bibitem{boev}
William Crawley-Boevey.
\newblock Locally finitely presented additive categories.
\newblock {\em Commun. Algebra}, 22(5):1641--1674, 1994.

\bibitem{EJ11}
Edgar~E Enochs and Overtoun~MG Jenda.
\newblock {\em Relative homological algebra: Volume 1}, volume~30.
\newblock Walter de Gruyter, 2011.

\bibitem{Esmkhani2007}
Mohammad~Ali Esmkhani and Massoud Tousi.
\newblock Gorenstein homological dimensions and {A}uslander categories.
\newblock {\em J. Algebra}, 308(1):321--329, 2007.

\bibitem{G01}
Aleksander~Aleksandrovich Gerko.
\newblock On homological dimensions.
\newblock {\em Sbornik: Mathematics}, 192(8):1165, 2001.

\bibitem{GT12}
R\"{u}diger G\"{o}bel and Jan Trlifaj.
\newblock {\em Approximations and endomorphism algebras of modules\textup{:} {V}olume 1 -- {A}pproximations}, volume~41 of {\em De Gruyter Expositions in Mathematics}.
\newblock Walter de Gruyter GmbH \& Co. KG, Berlin, second revised and extended edition, 2012.

\bibitem{Gov65}
VE~Govorov.
\newblock On flat modules.
\newblock {\em Sibirskii Matematicheskii Zhurnal}, 6(2):300--304, 1965.

\bibitem{HHLG24}
Dolors Herbera, Michal Hrbek, and Giovanna~Le Gros.
\newblock Cotorsion pairs and tor-pairs over commutative noetherian rings.
\newblock {\em arXiv preprint arXiv:2411.04514}, 2024.

\bibitem{holm}
Henrik Holm.
\newblock Gorenstein homological dimensions.
\newblock {\em J. Pure Appl. Algebra}, 189(1-3):167--193, 2004.

\bibitem{Hol17}
Henrik Holm.
\newblock The structure of balanced big {C}ohen--{M}acaulay modules over {C}ohen--{M}acaulay rings.
\newblock {\em Glasgow Mathematical Journal}, 59(3):549--561, 2017.

\bibitem{HJ06}
Henrik Holm and Peter J{\o}rgensen.
\newblock Semi-dualizing modules and related {G}orenstein homological dimensions.
\newblock {\em Journal of Pure and Applied Algebra}, 205(2):423--445, 2006.

\bibitem{HJ07}
Henrik Holm and Peter J{\o}rgensen.
\newblock {C}ohen--{M}acaulay homological dimensions.
\newblock {\em Rendiconti del Seminario Matematico della Universit{\`a} di Padova}, 117:87--112, 2007.

\bibitem{HJ11}
Henrik Holm and Peter J{\o}rgensen.
\newblock Rings without a {G}orenstein analogue of the {G}ovorov--{L}azard theorem.
\newblock {\em Quarterly journal of mathematics}, 62(4):977--988, 2011.

\bibitem{hrbek2023finite}
Michal Hrbek and Giovanna Le~Gros.
\newblock The finite type of modules of bounded projective dimension and {S}erre's conditions.
\newblock {\em Bulletin of the London Mathematical Society}, 56(8):2760--2775, 2024.

\bibitem{J03}
Peter J{\o}rgensen.
\newblock Recognizing dualizing complexes.
\newblock {\em Fundamenta Mathematicae}, 176:251--259, 2003.

\bibitem{Kra01}
Henning Krause.
\newblock The spectrum of a module category.
\newblock {\em Mem. Amer. Math. Soc.}, 149(707):x+125, 2001.

\bibitem{Laz69}
Daniel Lazard.
\newblock Autour de la platitude.
\newblock {\em Bulletin de la Soci{\'e}t{\'e} Math{\'e}matique de France}, 97:81--128, 1969.

\bibitem{Lenzing83}
H.~Lenzing.
\newblock Homological transfer from finitely presented to infinite modules.
\newblock In {\em Abelian group theory ({H}onolulu, {H}awaii, 1983)}, volume 1006 of {\em Lecture Notes in Math.}, pages 734--761. Springer, Berlin, 1983.

\bibitem{miyachi1998cohen}
Jun-ichi Miyachi.
\newblock Cohen-macaulay approximations and noetherian algebras.
\newblock {\em Communications in Algebra}, 26(7):2181--2190, 1998.

\bibitem{resol}
Leonid Positselski.
\newblock Resolutions as directed colimits.
\newblock {\em Rendiconti del Seminario Matematico della Universit{\`a} di Padova}, 2024.

\bibitem{SSY20}
Parviz Sahandi, Tirdad Sharif, and Siamak Yassemi.
\newblock Cohen--{M}acaulay homological dimensions.
\newblock {\em Mathematica Scandinavica}, 126(2):189--208, 2020.

\bibitem{shaul2023acyclic}
Liran Shaul.
\newblock Acyclic complexes of injectives and finitistic dimensions.
\newblock {\em arXiv preprint arXiv:2303.08756}, 2023.

\bibitem{Tak08}
Ryo Takahashi.
\newblock On {G}-regular local rings.
\newblock {\em Communications in Algebra{\textregistered}}, 36(12):4472--4491, 2008.

\bibitem{xu}
Jinzhong Xu.
\newblock {\em Flat covers of modules}, volume 1634 of {\em Lecture Notes in Mathematics}.
\newblock Springer-Verlag, Berlin, 1996.

\end{thebibliography}

\end{document}